\documentclass[12pt, oneside]{amsart}       % use "amsart" instead of "article" for AMSLaTeX format
\usepackage{geometry}                        % See geometry.pdf to learn the layout options. There are lots.
\usepackage[T1]{fontenc}
\usepackage{palatino}

\usepackage{fancyhdr}
\usepackage{graphicx}                % Use pdf, png, jpg, or epsÂ§ with pdflatex; use eps in DVI mode
                            \usepackage{amsthm}    
\newtheorem{theorem}{Theorem}[section]

\newtheorem{definition}{Definition}[section]
\newtheorem{lemma}[theorem]{Lemma}
\theoremstyle{remark}
\newtheorem*{remark}{Remark}
\usepackage{tcolorbox}

\title{ non-archimedean Fr\'{e}chet algebras and the loop space of a hypersurface complement}

\author{Emile Bouaziz}

\begin{document} \maketitle \begin{abstract} We study the space, $L\mathbf{A}^{d}_{f}$, of loops into a hypersurface complement, and show that the corresponding topological algebra of Laurent series with coefficients in $\mathcal{O}(L\mathbf{A}^{d}_{f})$ is a topological localisation of $\mathcal{O}(L\mathbf{A}^{d})$. This requires introducing a small amount of non-Archimedean functional analysis. In particular  we work with topological algebras whose topology is generated by a family of sub-multiplicative, non-Archimedean semi-norms. \end{abstract}

\section{introduction} If $X$ is a variety over a field, $k$, of characteristic $0$, then we write $LX$ for its space of \emph{algebraic loops}. We have by definition $LX(R)=X(R((z)))$. $LX$ can be thought of as roughly the mapping space $\operatorname{Map}(D^{*},X)$, where $D^{*}:=\operatorname{spec}(k((z)))$ is the \emph{formal punctured disc}, and as such is an algebraic analogue of the space $\operatorname{Map}(S^{1},M)$ of smooth loops into a manifold. The space $LX$ has been studied by numerous authors from a variety of perspectives, we mention here the works \cite{Dr}, \cite{KV} and \cite{GaiRoz} in particular. A familiar feature is the presence of topologies on the algebras of functions under consideration, notably the so-called \emph{Tate} topological $k$-vector space $k((z))$, cf. \cite{Dr} for a discussion of Tate objects in algebraic geometry.

 If $X$ is affine then $LX$ is representable by an ind-affine scheme with transition maps closed embeddings. The algebra $\mathcal{O}(LX)$ is thus naturally a topological $k$-algebra with a basis of open neighbourhoods at $0$ consisting of ideals. We can form then the algebra $\mathcal{O}(LX)\{z\}$ consisting of two way infinite Laurent series with Laurent tails tending to $0$ topologically. There is a natural morphism $\operatorname{ev}_{X}:\mathcal{O}(X)\rightarrow\mathcal{O}(LX)\{z\}$ which for $X=\mathbf{A}^{1}$ corresponds to the \emph{universal Laurent series}. We can think of this, informally, as corresponding geometrically to an evaluation map $\operatorname{Map}(D^{*},X)\times D^{*}\rightarrow X$. 

Our goal in this note is to show that, in a sense that we will make precise, this mapping correspondence is surprisingly local on $X$, at least when $X$ is open inside an affine space $\mathbf{A}^{d}$ (although we conjecture it is true for smooth varieties generally). More precisely we will prove, writing $\mathbf{A}^{d}_{f}$ for the complement of the hypersurface $f=0$, the following theorem. \begin{theorem}$\mathcal{O}(L\mathbf{A}^{d}_{f})\{z\}$ is obtained as the Cauchy completion of the localisation of $\mathcal{O}(L\mathbf{A}^{d})\{z\}$ at the element $\operatorname{ev}_{\mathbf{A}^{d}}(f).$\end{theorem}  In order to do this we need to find a convenient category of topological $k$-algebras to which $\mathcal{O}(L\mathbf{A}^{d})\{z\}$ belongs in which to form the localisation. This is why we will work with non-Archimedean Fr\'{e}chet algebras. Let us note that this result stands in stark contrast to the case of $\mathcal{O}(LX)$ in place of $\mathcal{O}(LX)\{z\}$.  Firstly we lack a candidate element at which to localise, as $LX$ does not naturally live over $X$ because the punctured disc $D^{*}$ has no $k$-points at which to evaluate a loop. Secondly, if $U$ is open in $X$, then $LU$ is very far from open in $LX$. Indeed $L\mathbf{A}^{1}$ is a vector space and $L\mathbf{G}_{m}$ is highly non-reduced. Note that in the work \cite{KV} a variant of the loop space is introduced which has some desirable properties that $LX$ lacks, nonetheless this variant does not have $k$-points $X(k((z)))$, and we will not deal with it in this note.

\section{Preliminary Notions}
\subsection{Loops.} We will now introduce the geometric objects of study. The initiated reader can safely skim this section so as to fix notation. \begin{definition} If $X$ is a $k$-scheme then the functor $LX:\mathbf{Alg}_{k}\rightarrow\mathbf{Sets}$ is defined by $LX(R):=X(R((z)))$ and referred to as the \emph{loop space} of $X$. \end{definition} We have the following standard lemma, the proof of which we include so as to fix some notation. \begin{lemma} If $X$ is a finite type $k$-scheme then $LX$ is representable by an ind-affine scheme which is a countable colimit along closed embeddings. \end{lemma} \begin{proof} Let $X$ be given as the zero locus of some polynomials $f_{1},...,f_{c}$ in variables $x^{1},...,x^{d}$. For each $n\geq 0$ consider variables $x^{i}_{j}$, for $i=1,...,d$ and $j\geq -n$. Write $x^{i}(z):=\sum_{j\geq -n}x_{j}z^{j}$ with $z$ a formal variable. Let us write $L^{n}X$ for the closed subset of the affine space with coordinates $x^{i}_{j}$ as before, subject to the relations implied by the equations of Laurent series $f_{l}(x^{1}(z),...,x^{d}(z))=0$, for $l=1,...,c$. $L^{n}X$ is naturally a closed subscheme of $L^{n+1}X$ and we form the ind-scheme $LX:=\operatorname{colim}L^{n}X$. It is easily seen that this represents the desired functor. \end{proof}\begin{remark}\begin{itemize}\item  Note that the schemes $L^{n}X$ do not have intrinsic meaning as functors of $X$, we had to choose an embedding to make sense of them. Nonetheless the colimit of course has intrinsic meaning, representing as it does the functor $LX$.\item In the particular case of $\mathbf{A}^{1}$ we have that $\mathcal{O}(L\mathbf{A}^{1})\cong\operatorname{lim}_{n}k[x_{i}\,|\, i\geq -n]$, which we will consider as a topological $k$-algebra in the obvious way.\end{itemize} \end{remark}

\begin{definition} Let $A$ be a topological $k$-algebra which is the limit of an inverse system $(A_{n})_{n}$ of topologically discrete algebras. The topological $k$-algebra $A\{z\}$ is defined as $\operatorname{lim}_{n} A((z))$, with the limit taken inside topological $k$-algebras. \end{definition}
\begin{remark} Elements of $A\{z\}$ are explicitly represented as two way infinite sums $\sum_{i}a_{i}z^{i}$, such that $a_{-i}\rightarrow 0$ as $i\rightarrow+\infty$. \end{remark}

\begin{definition} \begin{itemize}\item For $X$ an affine $k$-scheme we define the topological $k$-algebra $\mathcal{A}_{X}$ to be $\mathcal{O}(LX)\{z\}$. \item The evaluation morphism $\operatorname{ev}_{\mathbf{A}^{1}}:\mathcal{O}(\mathbf{A}^{1})\rightarrow\mathcal{A}_{\mathbf{A}^{1}}$ is defined by sending $x\in \mathbf{A}^{1}$ to $x(z):=\sum_{i}x_{i}z^{i}$. An identical formula defines $\operatorname{ev}_{X}:\mathcal{O}(X)\rightarrow\mathcal{A}_{X}$ for each affine $X$, noting that the relations defining the algebras are obviously preserved. \end{itemize}\end{definition}

\subsection{Fr\'{e}chet Alegrbas.} We wish to study $\mathcal{A}_{X}$ as a topological algebra. Let us note that is \emph{non-Archimedean}, which is to say the underlying topological $k$-vector space has a basis of open neighbourhoods at $0$ consisting of subspaces. Further more, the underlying topological $k$-vector space is Cauchy complete. Equivalently, an infinite sum $\sum_{i}a_{i}$ is convergent iff $a_{i}\rightarrow 0$. We want to be able to perform some familiar algebraic operations on the algebra $\mathcal{A}_{X}$, in particular to localize at an element of the image of $\operatorname{ev}_{X}$. We will see how this can be done in a somewhat ad-hoc manner, that is nonetheless good enough for our purposes, using non-Archimedean Fr\'{e}chet algebras over $k$. Of course the definition we present is heavily inspired by the corresponding notion in classical functional analysis, as well as some well known constructions in non-Archimedean analysis, cf. \cite{Ber} for example. Note that we claim no real originality in these definitions, our goal is simply to introduce a context adequate for our purposes. 

 \begin{definition} A non-Archimedean $\mathbf{R}$ semi-norm topological algebra $A$ is a continuous sub- multiplicative map $|-|:A\rightarrow \mathbf{R}_{\geq 0}$ so that $|a+b|\leq\operatorname{max}\{|a|,|b|\}.$ We refer to these as semi-norms henceforth, and leave implicit that they are non-Archimedean. A family $|-|_{i\in I}$ of semi-norms on $A$ is said to \emph{induce} the topology on $A$ if a sequence $a_{n}\rightarrow 0$ iff $|a_{n}|_{i}\rightarrow 0$ for all $i\in I$. \end{definition} \begin{remark}Note that by \emph{sub-multiplicative} we simply mean that $|fg|\leq|f||g|$. We require sub-multipicativity as we wish to deal with topological algebras such as $A((z))$ where $A$ is possibly not an integral domain, so that the valuation of the product of two elements can potentially be greater than the sum of the valuations of these elements.\end{remark}

\begin{definition} A complete topological $k$ algebra is called a non-Archimedean Fr\'{e}chet algebra ($\mathbf{nAF}_{k}$ henceforth) if it has a countable family of non-Archimedean semi-norms inducing its topology.\end{definition} 

 \begin{lemma} For $X$ a finite type affine scheme the topological $k$-algebra $\mathcal{A}_{X}$ is an element of $\mathbf{nAF}_{k}$. \end{lemma} \begin{proof} We write $\mathcal{O}(LX)$ as a countable limit, $\lim_{n}A_{n}$, of topologically discrete algebras. We have for each $n$, a norm on $A_{n}((z))$ which induces its topology by fixing $0<\varepsilon<1$ and setting $|a|_{n}:=\varepsilon^{\operatorname{val}_{z}(a)}$, where $\operatorname{ord}_{z}$ denotes $z$-adic order. This induces a family of $\mathbf{R}$ norms on $\mathcal{A}_{X}$ which are easily seen to induce its topology, whence we have exhibited $\mathcal{A}_{X}$ as an element of $\mathbf{nAF}_{k}$. \end{proof}

\begin{definition} An element $a\in A$, for $A$ in $\mathbf{nAF}_{k}$, is called \emph{good} if for all $i\in I$ we have $|a|_{i}\neq 0$ and if further $|a|_{i}^{n}=|a^{n}|_{i}$ for all $i$ and $n$. If this is the case, each of the semi-norms extends uniquely to the usual algebraic localisation, which we denote $A[a^{-1}]^{\operatorname{alg}}$. The completion of this with respect to the topology induced by the family of  extended norms will be denoted $A[a^{-1}]\in\mathbf{nAF}_{k}$. \end{definition} \begin{remark} We note that (suppressing the index $i$) we require $|a|\neq 0$ so that we can define $|ba^{-n}|:=|b||a|^{-n}$ and we require $|a^{n}|=|a|^{n}$ so that the extended norms remain sub-multiplicative, which ensures for example the continuity of multiplication.\end{remark}

\begin{lemma} For any $f\in\mathcal{O}(\mathbf{A}^{d})$, the element $\operatorname{ev}_{\mathbf{A}^{d}}(f)\in\mathcal{O}(L\mathbf{A}^{d})$ is good in the sense of definition 2.6. \end{lemma}
\begin{proof} $\mathcal{O}(L\mathbf{A}^{d})$ is the pro-limit of discrete integral domains, so the multiplicativity of norms is immediate. It remains to check that $\operatorname{ev}_{\mathbf{A}^{d}}(f)$ does not vanish modulo the ideal generated by elements $x_{i}$ with $i\leq -n$, which is clear. \end{proof}

We now identify a class of morphisms at which localisations are well behaved. 

\begin{definition} A morphism $f:A\rightarrow B$ is called bounded above (resp. below) if for all  $i\in I$ there exists $C_{i}>0$ so that $|f(a)|_{i}$ is at most (resp. at least) $C_{i}|a|_{i}$. If a morphism is bounded above and below we call it a quasi-isometric embedding, and if we can take $C_{i}=c_{i}=1$ for all $i$ then it is called an isometric embedding. \end{definition}

\begin{remark} A quasi-isometric embedding is necessarily injective, justifying the name. Further note that the induced topology on the image agrees with the topology on the domain. \end{remark}

\begin{lemma} If $f:A\rightarrow B$ is a quasi-isometric embedding in $\mathbf{nAF}_{k}$, $a\in A$ is good and $f(a)$ is a unit of $B$, then there is a unique continuous extension of $f$ to $A[a^{-1}]\rightarrow B$. Further, $A[a^{-1}]$ agrees with the completion of $A[a^{-1}]^{\operatorname{alg}}$ inside $B$, with respect to the induced topology. \end{lemma} \begin{proof} First note that $f(a)$ is automatically good, so behaves multiplicatively with respect to norms. We need to check that $a_{n}a^{-n}$ tends to zero, i.e. for all $i\in I$, $a_{n}a^{-n}$ tends to zero with respect to the semi-norm $|-|_{i}$, iff $f(a_{n}a^{-i}):=f(a_{n})f(a)^{-n}$ tends to zero. Again this can be checked on semi-norms, and then the quasi-isometry assumption immediately implies the claim.\end{proof}

\begin{lemma} If $A\rightarrow B$ is a map of pro-discrete algebras, which is a limit of injective maps $A_{j}\rightarrow B_{j}$, for discrete $A_{j}$ and $B_{j}$, then the map in $\mathbf{nAF}_{k}$, $A\{z\}\rightarrow B\{z\}$ is a quasi-isometric embedding. \end{lemma} \begin{proof} Some thought confirms shows that it suffices to prove this for a map $A((z))\rightarrow B((z))$, induced from an injection of discrete algebras $A\rightarrow B$. This is now immediate, indeed we can take the constants $C$ and $c$ to be $1$, and we have an isometric embedding. \end{proof}

\section{Proof of main theorem} We are now in a position to formulate and prove the main theorem of this note. We recall here that we write $\mathbf{A}^{d}_{f}$ for the complement of the hypersurface $f=0$ inside $\mathbf{A}^{d}$. We let $x^{1},...,x^{d}$ denote coordinates on the affine space $\mathbf{A}^{d}$. We consider $\mathbf{A}^{d}_{f}$ as emebedded into $\mathbf{A}^{d+1}$, with coordinates $x^{1},...,x^{d},y$, as the vanishing locus of $f(x^{1},...,x^{d})y=1$. 

We begin with the following lemma, which is where we crucially have to work with $\mathbf{A}^{d}$ in place of a general smooth variety. \begin{lemma} The morphism $\mathcal{A}_{\mathbf{A}^{d}}\rightarrow\mathcal{A}_{\mathbf{A}^{d}_{f}}$ is a quasi-isometric embedding in $\mathbf{nAF}_{k}$. \end{lemma}\begin{proof} By lemma 2.4. above, it suffices to show that the map $\mathcal{O}(L\mathbf{A}^{d})\rightarrow\mathcal{O}(L\mathbf{A}^{d}_{f})$ can be written as the limit of a projective system consisting of injections of topologically discrete algebras. As usual we consider $\mathcal{O}(L\mathbf{A}^{d})$ as the limit of discrete algebras $A_{n}:=k[x^{i}_{j}\,|\,i=1,...,d, \, j\geq -n]$. We write then $B_{n}:=A_{n}[y_{i}\,|\, i\geq -n]/(\operatorname{relations})$, where the relations are those implied by the identity of Laurent series $y(z)f(x^{1}(z),...,x^{d}(z))=1$, where as usual we have written $y(z):=\sum y_{i}z^{i}$. Note that $\mathcal{O}(L\mathbf{A}^{d}_{f})\cong\operatorname{lim}B_{n}$. 

Now a morphism $B_{n}\rightarrow C$ is equivalent to elements $c^{i}(z)\in z^{-n}C[[z]]$, for $i=1,...,d$, along with an inverse to $f(c^{1}(z),...,c^{d}(z))$ which also lies in $z^{-n}C[[z]]$. Let $f_{\operatorname{min}}\in A_{n}$ be the coefficient of the highest power of $z^{-1}$ occuring in $f(x^{1}(z),...,x^{d}(z))$, and note that $f_{\operatorname{min}}\neq 0$. Further let $C_{n}:=A_{n}[f_{\operatorname{min}}^{-1}]$ and note that the map $A_{n}\rightarrow C_{n}$ is an injection as $A_{n}$ is obviously an integral domain. Now we observe that there is a map $B_{n}\rightarrow C_{n}$ which factors the inclusion $A_{n}\rightarrow C_{n}$, indeed $f(x^{1}(z),...,x^{d}(z))$ is invertible in $C_{n}((z))$ because its lowest order term is. We deduce thus that $A_{n}\rightarrow B_{n}$ is an inclusion and thus the lemma is proven. \end{proof}

The main theorem now follows from a pleasant trick involving differential operators on the punctured disc $D^{*}$, the essential idea is that differential operators are dense inside all $k$-linear continuous endomorphisms of $\mathcal{O}(D^{*})=k((z))$. This fact is presumably well known, but we essentially give a proof below anyway.

\begin{theorem}  The natural map of non-Archimedean Fr\'{e}chet algebras, $\mathcal{A}_{\mathbf{A}^{d}}[\operatorname{ev}_{\mathbf{A}^{d}}(f)^{-1}]\rightarrow\mathcal{A}_{\mathbf{A}^{d}_{f}}$, is an equivalence.\end{theorem} 

\begin{proof} We know by lemmas 2.4. and 3.1. that the topology on $\mathcal{A}_{\mathbf{A}^{d}}[\operatorname{ev}_{\mathbf{A}^{d}}(f)^{-1}]$ agrees with the induced topology for the inclusion into $\mathcal{A}_{\mathbf{A}^{d}_{f}}$. As such we need only show that the image is dense in $\mathcal{A}_{\mathbf{A}^{d}_{f}}$. 

In order to prove this let us first note that the algebra, $\operatorname{Diff}_{D^{*}}$, of differential operators on $D^{*}$ acts continuously on $\mathcal{A}_{X}$, simply by acting on powers of $z$ in the evident manner. We claim now that the sub-algebra spanned by the $\operatorname{Diff}_{D^{*}}$ span of the image of $\operatorname{ev}_{X}$ is dense in $\mathcal{A}_{X}$ for any affine finite type $X$. Indeed it will follow immediately from the case of $\mathbf{A}^{1}$, which we now prove. It suffices to show that each topological $x_{i}\in\mathcal{A}_{\mathbf{A}^{1}}$ is in the desired topological closure.  

We write $\nu:=z\partial_{z}$ and define, for each $j$, a sequence of differential operators, $(P^{j}_{n})_{n}$ as follows. We set $P^{j}_{n}:=\lambda^{j}_{n}\prod_{i\in[-n,n]\setminus\{j\}}(\nu-i)$, where $\lambda^{j}_{n}$ are chosen so that $P^{j}_{n}(z^{j})=z^{j}$. We claim now that $P^{j}_{n}(x(z))\rightarrow x_{j}z^{j}$, as $n\rightarrow \infty$. Indeed let us note that the difference $P^{j}_{n}(x(z))-x_{j}z^{j}$ is of the form $\sum_{i}c_{in}x_{i}z^{i}$ for some scalars $c_{in}$, so that $c_{in}=0$ for all $i\in [-n,n]$. It follows that this is of the form $I_{\leq n}+z^{n}\mathcal{O}(L\mathbf{A}^{1})[[z]]$, where $I_{\leq n}$ is the ideal generated by the coordinates $x_{i}$ with $i\leq n$. This difference of course tends to $0$, whence we are done.

Now to conclude it suffices to show that the image of $\mathcal{A}_{\mathbf{A}^{d}}[\operatorname{ev}_{\mathbf{A}^{d}}(f)^{-1}]$ inside $\mathcal{A}_{\mathbf{A}^{d}_{f}}$ contains the image of $\operatorname{ev}_{\mathbf{A}^{d}_{f}}$ and is stable under $\operatorname{Diff}_{D^{*}}$. For this it suffices to check that it is stable under $\partial_{z}$, and it is clear how to extend the derivation to the localisation, so we are done. \end{proof}

\begin{remark} We expect this result to be true for smooth affine schemes more generally. It is lemma 3.1. which we have been unable to prove in this context. \end{remark}


\begin{thebibliography}{9}

\bibitem{Ber}
V. Berkovich,
\textit{Spectral Theory and Analytic Geometry over Non-Archimedean Fields,}
Mathematical surveys and monographs, no. 33. (AMS.)

\bibitem{Dr}
V. Drinfeld,
\textit{Infinite Dimensional Vector Bundles in Algebraic Geometry,}
The Unity of Mathematics. Progress in Mathematics, vol 244. Birkhauser Boston.

\bibitem{GaiRoz}
D. Gaitsgory, N. Rozenblyum,
\textit{DG indschemes,}
 Perspectives in Representation Theory (Contemporary Mathematics 610)

\bibitem{KV}
M. Kapranov, E. Vasserot,
\textit{Vertex Algebras and Formal Loop Space,}
Publications Math\'{e}matiques de l'IH\'{E}S, Tome 100 (2004), pp. 209-269.


\end{thebibliography}
\end{document}